\documentclass{amsart}

\usepackage[latin1]{inputenc}
\usepackage[english]{babel}

\usepackage{indentfirst}
\usepackage{amssymb}
\usepackage{amsthm}

\newcommand{\FF}{\protect{\mathcal F}}
\newcommand{\DD}{\protect{\mathcal D}}

\newcommand{\vf}{\varphi}

\def\cC{{{\mathcal C}}}

\def\cF{{{\mathcal F}}}

\newcommand\eps{\ensuremath{\varepsilon}}

\newcommand{\impli}{\Rightarrow}
\newcommand{\Nat}{\mathbb{N}}

\newcommand{\erre}{\mathbb{R}}

\newtheorem{theo}{Theorem}[section]
\newtheorem{lem}[theo]{Lemma}
\newtheorem{pro}[theo]{Proposition}
\newtheorem{cor}[theo]{Corollary}
\newtheorem{defi}[theo]{Definition}
\newtheorem{rem}[theo]{Remark}

\newtheorem{prop}[theo]{Proposition}

\def\epsilon{\varepsilon}

\providecommand{\bysame}{\leavevmode\hbox to3em{\hrulefill}\thinspace}
\providecommand{\MR}{\relax\ifhmode\unskip\space\fi MR }
\providecommand{\MRhref}[2]{%
  \href{http://www.ams.org/mathscinet-getitem?mr=#1}{#2}
}
\providecommand{\href}[2]{#2}

\title{The McShane integral in weakly compactly generated spaces}

\author{ A. Avil\'{e}s}
\address{Departamento de Matem\'{a}ticas\\
Facultad de Matem\'{a}ticas\\ Universidad de Murcia\\ 30100 Espinardo (Murcia)\\
Spain} \email{avileslo@um.es}

\author{G. Plebanek}
\address{Mathematical Institute\\ University of Wroc\l aw\\ Pl.\ Grunwaldzki 2/4\\
50-384 Wroc\-\l aw\\ Poland} \email{grzes@math.uni.wroc.pl}

\author{J. Rodr\'{i}guez}
\address{Departamento de Matem\'{a}tica Aplicada\\
Facultad de Inform\'{a}tica\\ Universidad de Murcia\\ 30100 Espinardo (Murcia)\\
Spain} \email{joserr@um.es}



\thanks{A.~Avil\'{e}s and J.~Rodr\'{i}guez were supported by MEC and FEDER (Project MTM2008-05396)
and Fundaci\'{o}n S\'{e}neca (Project 08848/PI/08). A. Avil\'{e}s was 
supported by Ramon y Cajal contract (RYC-2008-02051). 
G.~Plebanek wishes to thank A.~Avil\'{e}s, B.~Cascales and J.~Rodr\'{i}guez
for their hospitality during his stay in Murcia in
February 2009; the visit was supported by Departamento de Matem\'{a}ticas,
Universidad de Murcia.}

\begin{document}

\begin{abstract}
Di Piazza and Preiss asked whether every Pettis integrable
function defined on $[0,1]$ and taking values in a weakly
compactly generated Banach space is McShane integrable. In this
paper we answer this question in the negative. Moreover, we give 
a counterexample where the target Banach space is reflexive.
\end{abstract}

\maketitle

\section{Introduction}

The classical Pettis' measurability theorem \cite{Pettis} ensures that scalar and strong measurability
are equivalent for functions taking values in separable Banach spaces. This fact has many interesting consequences in vector integration. For instance, it is a basic tool to prove that Pettis and McShane integrability coincide in separable Banach spaces, \cite{fre5,fre-men,gor}. However, for non-separable Banach spaces the notions of scalar and strong measurability are different in general. This leads to subtle problems when trying to compare different types of integrals.

In this paper we deal with the Pettis and McShane integrals. Di
Piazza and Preiss~\cite{dip-pre} asked whether every Pettis
integrable function $f:[0,1]\to X$ is McShane integrable if $X$ is
a weakly compactly generated (WCG) Banach space. Recently, Deville
and the third author~\cite{dev-rod} have proved that the answer is
affirmative when $X$ is Hilbert generated, thus improving the
previous results obtained in~\cite{dip-pre,Rod10}. Our main
purpose here is to show that the question of Di Piazza and Preiss
has negative answer in general.

The paper is organized as follows.

In Section~\ref{SectionMcShane2} we introduce the MC-integral for Banach space-valued functions defined on probability spaces. This auxiliary tool is used as a substitute of the McShane integral at some stages. We prove that, for functions defined on quasi-Radon probability spaces, MC-integrability always implies McShane integrability (Proposition~\ref{Mc2impliesMc}), while the converse holds if the topology on the domain has a countable basis (Proposition~\ref{McimpliesMc2}). This approach allows us to give a partial answer (Corollary~\ref{CorollaryFremlinProblemMcShane}) to a question posed by Fremlin in \cite[4G(a)]{fre5}.

In Section~\ref{SectionMeasureFilling} we show that the existence of scalarly null (hence Pettis integrable) WCG-valued functions which are not McShane integrable is strongly related to the existence of families of finite sets which are ``measure filling'' in the sense of the following definition. Throughout the paper
$(\Omega,\Sigma,\mu)$ is a probability space and we use the symbol $[S]^{<\omega}$
to denote the family of all finite subsets of a given set~$S$.

\begin{defi}\label{DefinitionMCfilling}
A family $\mathcal{F} \subset [\Omega]^{<\omega}$ is called {\em MC-filling} on~$\Omega$ if it is hereditary
and there exists $\eps>0$ such that for every countable partition $(\Omega_m)$ of $\Omega$ there is $F\in\mathcal{F}$ such that
$$
    \mu^\ast\left( \bigcup\{\Omega_m : \, F\cap \Omega_m\neq\emptyset \}\right) > \varepsilon,
$$
where $\mu^\ast$ is the outer measure induced by~$\mu$.
\end{defi}

This concept should be viewed as a measure-theoretic analogue of
the notion of $\varepsilon$-filling families, defined as follows:

\begin{defi}\label{DefinitionFilling}
Let $\varepsilon>0$. A family~$\mathcal{F}\subset [S]^{<\omega}$ is called
{\em $\varepsilon$-filling}
on the set~$S$ if it is hereditary and for every
$H\in [S]^{<\omega}$ there is $F\in\mathcal{F}$ with $F\subset H$
and  $|F|\geq\varepsilon|H|$.
\end{defi}

The following problem by Fremlin~\cite{freDU} remains open:

\medskip
\noindent {\bf Problem DU.}
{\em Are there compact $\varepsilon$-filling families on uncountable sets?} 

\medskip
However, we show that
compact MC-filling families on $[0,1]$ can be constructed from some weaker versions of filling families that Fremlin proved
to exist (Theorem~\ref{PropositionGrzegorz}). This leads to our
main result:

\medskip
{\noindent {\bfseries Theorem~\ref{TheoremDP}}.\it{}
There exist a WCG Banach space~$X$ and a scalarly null function
$f:[0,1]\to X$ which is not McShane integrable.
}

\medskip
In fact, the space~$X$ can be taken reflexive (Theorem~\ref{CorollaryReflexive}). Observe that Theorem~\ref{TheoremDP} also answers in the negative the question (attributed to Musia{\l} in~\cite{dip-pre}) whether every scalarly null Banach space-valued function on~$[0,1]$ is McShane integrable. In \cite{dip-pre,Rod10} two counterexamples had been constructed under the Continuum Hypothesis
(where the target spaces are non-WCG spaces). We emphasize that all results in this paper are in ZFC.

In Section~\ref{SectionFilling} we prove that if a family $\cF \subset [A]^{<\omega}$ 
is $\varepsilon$-filling on a set $A \subset \Omega$ of positive outer measure then it
is MC-filling on~$\Omega$.

Finally, in Section~\ref{SectionRemarks} we provide an example of
a McShane integrable function which is not MC-integrable
(Theorem~\ref{PropositionMcShaneNoMcShane2}). Our example also
makes clear that, in general, the results on the coincidence of
Pettis and McShane integrability of~\cite{dev-rod,dip-pre} do not
hold when McShane integrability is replaced by MC-integrability.

\subsection*{Terminology}
Our standard references are \cite{die-uhl-J,tal} (vector integration) and \cite{FreMT-4} (topological measure theory).
By a \emph{partition} of a set~$S$ we mean a collection of pairwise disjoint (maybe empty) subsets whose union is~$S$. 
We say that a set is \emph{countable} if it is either finite or countably infinite.
The symbol $|S|$ stands for the cardinality of a set~$S$. A family $\mathcal{F} \subset [S]^{<\omega}$ is 
called \emph{hereditary} if $G\in \cF$ whenever $G\subset F \in \cF$. A family $\cF \subset [S]^{<\omega}$ is called \emph{compact} if the
set $\{1_A:A\in \cF\}$ is compact in~$2^S$ equipped with the product topology. 
Here we write $1_A$ to denote 
the function on~$S$ defined by $1_A(s)=1$ if $s\in A$ and $1_A(s)=0$ if $s\not\in A$.
It is well known that a hereditary family $\FF \subset [S]^{<\omega}$ is compact if and only if 
every set $A\subset S$ with $[A]^{<\omega} \subset \FF$ is finite.
We say that a family $\mathcal{E} \subset \Sigma$ is \emph{$\eta$-thick}
(for some $\eta>0$) if $\mu(\Omega \setminus \bigcup \mathcal{E})\leq \eta$.

Throughout the paper $X$ is a (real) Banach space. The norm of~$X$ is denoted by~$\|\cdot\|$ if it is needed explicitly.
We denote by $X^*$ the topological dual of~$X$ and put $B_{X}=\{x\in X: \|x\|\leq 1\}$. The space $X$ is called
\emph{WCG} if there is a weakly compact subset of~$X$ whose linear span is dense in~$X$. A
function $f:\Omega \to X$ is called \emph{scalarly null} if, for each $x^*\in X^*$, the composition $x^*f:\Omega \to \erre$ vanishes $\mu$-a.e.
(the exceptional set depending on~$x^*$).

If $\mathfrak{T}\subset \Sigma$ is a topology
on~$\Omega$, we say that $(\Omega,\mathfrak{T},\Sigma,\mu)$ is a \emph{quasi-Radon} probability space (following
\cite[Chapter~41]{FreMT-4}) if $\mu$ is complete, inner regular with respect to closed sets, and $\mu(\bigcup \mathcal{G})=\sup\{\mu(G):
G\in \mathcal{G}\}$ for every upwards directed family $\mathcal{G}\subset
\mathfrak{T}$. A \emph{gauge} on~$\Omega$ is a function $\delta:\Omega \to \mathfrak{T}$ such that
$t\in \delta(t)$ for all $t\in \Omega$. Every Radon probability space is quasi-Radon \cite[416A]{FreMT-4}.

The vector-valued McShane integral was first studied in \cite{fre-men,gor} for functions defined on~$[0,1]$ equipped with the Lebesgue measure. Fremlin~\cite{fre5} extended the theory  to deal with functions defined on arbitrary quasi-Radon probability spaces. We next recall an alternative way of defining the McShane integral taken from \cite[Proposition~3]{freET}.

\begin{defi}\label{DefinitionMcShane}
Suppose $(\Omega,\mathfrak{T},\Sigma,\mu)$ is quasi-Radon. A function $f:\Omega \to X$ is said to be
{\em McShane integrable}, with integral $x\in X$, if for every $\epsilon>0$ there exist $\eta>0$ and a gauge $\delta$ on~$\Omega$
such that: for every $\eta$-thick finite family $(E_i)$ of pairwise disjoint measurable sets and
every choice of points $t_i\in \Omega$ with $E_i \subset \delta(t_i)$, we have
$$
    \Bigl\|\sum_i \mu(E_{i})f(t_i)-x\Bigr\| \leq \epsilon.
$$
\end{defi}

Every McShane integrable function is also Pettis
integrable (and the corresponding integrals coincide), \cite[1Q]{fre5}.
The converse does not hold in general, see \cite{dev-rod,dip-pre,fre-men,Rod10} for examples.

\section{Another look at the McShane integral}\label{SectionMcShane2}

We next introduce a variant of the McShane integral that is defined in terms of the measure space only,
without any reference to a topology.

\begin{defi}\label{DefinitionMCIntegral}
A function $f:\Omega \to X$ is called {\em MC-integrable}, with integral $x\in X$, if for every
$\epsilon>0$ there exist $\eta>0$, a countable partition $(\Omega_m)$ of $\Omega$ and sets $A_m\in \Sigma$ with $\Omega_m\subset A_m$, such that: for every $\eta$-thick finite family $(E_i)$ of pairwise disjoint elements of~$\Sigma$ with $E_i\subset A_{m(i)}$ and
every choice of points $t_i\in \Omega_{m(i)}$, we have
$$
    \Bigl\|\sum_i \mu(E_i)f(t_i)-x\Bigr\| \leq \epsilon.
$$
\end{defi}

Clearly, given $\eta>0$, a countable partition $(\Omega_m)$ of~$\Omega$ and sets $\Omega_m  \subset A_m\in \Sigma$, we can always find families $(E_i)$ as in Definition~\ref{DefinitionMCIntegral}. It is routine to check that the vector~$x$ in Definition~\ref{DefinitionMCIntegral} is unique.

The relationship between the MC-integral and the McShane integral is analyzed in the following two propositions.

\begin{prop}\label{Mc2impliesMc}
Suppose $(\Omega,\mathfrak{T},\Sigma,\mu)$ is quasi-Radon.
If $f:\Omega\to X$ is MC-integrable, then it is McShane integrable (and the corresponding integrals coincide).
\end{prop}
\begin{proof}
Let $x \in X$ be the MC-integral of $f$ and fix $\varepsilon>0$. Since $f$ is MC-integrable, there exist $\eta>0$,
a countable partition $(\Omega_m)$ of $\Omega$ and measurable sets $A_m\supset \Omega_m$ satisfying the condition of Definition~\ref{DefinitionMCIntegral}.

For each $m,n\in \Nat$, set $\Omega_{m,n}:=\{t\in \Omega_m: n-1\leq \|f(t)\|<n\}$ and choose open $U_{m,n}\supset A_m$ such that
$$
    \mu(U_{m,n}\setminus A_m)\leq \frac{1}{2^{m+n}}\min\Bigl\{\frac{\epsilon}{n},\frac{\eta}{2}\Bigr\}.
$$
Clearly, $(\Omega_{m,n})$ is a partition of~$\Omega$. Define a gauge $\delta:\Omega \to \mathfrak{T}$ by $\delta(t):=U_{m,n}$ if $t\in \Omega_{m,n}$.
Let $(E_i)$ be a $\frac{\eta}{2}$-thick finite family of pairwise disjoint measurable sets and
let $t_i\in \Omega$ be points such that $E_i \subset \delta(t_i)$. We will check that
\begin{equation}\label{target}
    \Bigl\|\sum_i \mu(E_i)f(t_i)-x \Bigr\|\leq 2\epsilon.
\end{equation}

For each $i$, let
$m(i),n(i) \in \Nat$ be such that $t_i\in \Omega_{m(i),n(i)}$. The
set $F_i:=E_i \cap A_{m(i)}$ is measurable, $F_i \subset A_{m(i)}$ and $t_i \in \Omega_{m(i)}$. The $F_i$'s are pairwise disjoint.
Since
$$
    E_i\setminus F_i=E_i \setminus A_{m(i)} \subset \delta(t_i) \setminus A_{m(i)}=U_{m(i),n(i)}\setminus A_{m(i)}
$$
we have
\begin{multline*}
 \mu\Bigl(\Omega \setminus \bigcup_{i}F_i\Bigr)=
 \mu\Bigl(\Omega \setminus \bigcup_{i}E_i\Bigr)+\mu\Bigl(\bigcup_{i}E_i \setminus \bigcup_{i}F_i\Bigr) \leq \\
 \leq \frac{\eta}{2}+\mu \Bigl(\bigcup_{m,n=1}^\infty U_{m,n}\setminus A_m\Bigr) \leq \frac{\eta}{2}+
 \sum_{m,n=1}^\infty\frac{\eta}{2^{m+n+1}}=\eta,
\end{multline*}
and so the family $(F_i)$ is $\eta$-thick. From the MC-integrability condition it follows that
\begin{equation}\label{partialtarget}
    \Bigl\|\sum_i \mu(F_i)f(t_i)-x\Bigr\|\leq \epsilon.
\end{equation}
For each $\tilde{m},\tilde{n}\in \Nat$, let $I(\tilde{m},\tilde{n})$ be the (maybe empty) set of all indexes~$i$ for which 
$m(i)=\tilde{m}$ and $n(i)=\tilde{n}$. Observe that
\begin{multline*}
    \sum_{i\in I(\tilde{m},\tilde{n})} \mu(E_i \setminus F_i)\|f(t_i)\| \leq  \sum_{i\in I(\tilde{m},\tilde{n})} 
	\mu(E_i\setminus F_i) \tilde{n} = \\ =
    \mu\Bigl(\bigcup_{i\in I(\tilde{m},\tilde{n})} E_i\setminus F_i \Bigr)\tilde{n} \leq 
	\mu(U_{\tilde{m},\tilde{n}}\setminus A_{\tilde{m}})\tilde{n} \leq
    \frac{\epsilon}{2^{\tilde{m}+\tilde{n}}}.
\end{multline*}
Therefore
\begin{multline}\label{partialtarget2}
 \Bigl\|\sum_i \mu(E_i)f(t_i)-\sum_i\mu(F_i)f(t_i)\Bigr\| \leq
 \sum_i \mu(E_i \setminus F_i) \|f(t_i)\| = \\ =
    \sum_{\tilde{m},\tilde{n}=1}^\infty \, \sum_{i\in I(\tilde{m},\tilde{n})}\mu(E_i \setminus F_i) \|f(t_i)\| \leq
    \sum_{\tilde{m},\tilde{n}=1}^\infty\frac{\epsilon}{2^{\tilde{m}+\tilde{n}}}=\epsilon.
\end{multline}

Inequality \eqref{target} now follows from \eqref{partialtarget} and~\eqref{partialtarget2}. This shows that $f$ is McShane integrable, with McShane integral~$x$.
\end{proof}

The converse of Proposition~\ref{Mc2impliesMc} does not hold in general (see Theorem~\ref{PropositionMcShaneNoMcShane2} below), although it is true for certain quasi-Radon spaces like $[0,1]$, as we next prove.

\begin{prop}\label{McimpliesMc2}
Suppose $(\Omega,\mathfrak{T},\Sigma,\mu)$ is quasi-Radon and $\mathfrak{T}$ has a countable basis.
Then $f:\Omega\to X$ is McShane integrable (if and) only if it is MC-integrable.
\end{prop}
\begin{proof} It only remains to prove the ``only if''. Assume that $f$ is McShane integrable, with McShane integral $x\in X$.
Let $\{U_m : m\in\mathbb{N}\}$ be a countable basis for~$\mathfrak{T}$
consisting of mutually distinct elements. 
Fix $\varepsilon>0$. Find $\eta>0$ and a gauge $\delta$ on~$\Omega$ fulfilling the condition of Definition~\ref{DefinitionMcShane}. We can suppose without loss of generality that $\delta(t)\in\{U_m : m\in\mathbb{N}\}$ for every $t\in \Omega$. Set
$$
    \Omega_m := \{t\in\Omega : \delta(t)=U_m\} \
    \mbox{ and }
    \ A_m:=U_m \quad \mbox{for all }\, m\in \Nat.
$$
Clearly, $(\Omega_m)$ is a partition of~$\Omega$ and $\Omega_m\subset A_m \in \Sigma$. Now let $(E_i)$ be an $\eta$-thick finite family of pairwise disjoint measurable sets with $E_i\subset A_{m(i)}$ and let $t_i\in \Omega_{m(i)}$. Then $\delta(t_i)=U_{m(i)}=A_{m(i)}$, hence $E_i \subset \delta(t_i)$ for all~$i$. From the McShane integrability condition it follows that
$\|\sum_i \mu(E_i)f(t_i)-x\| \leq \epsilon$. This shows that $f$ is MC-integrable.
\end{proof}

Fremlin raised in \cite[4G(a)]{fre5} the following question: \emph{Does any topology on~$\Omega$ for which $\mu$ is quasi-Radon yield the same collection of McShane integrable $X$-valued functions?} In view of Propositions \ref{Mc2impliesMc} and~\ref{McimpliesMc2}, we get a partial answer:

\begin{cor}\label{CorollaryFremlinProblemMcShane}
Let $\mathfrak{T}_1$ and $\mathfrak{T}_2$ be two topologies on~$\Omega$ for which $\mu$ is quasi-Radon. Suppose $\mathfrak{T}_1$ has a countable basis. If $f:\Omega \to X$ is McShane integrable with respect to~$\mathfrak{T}_1$, then it is also McShane integrable with respect to~$\mathfrak{T}_2$ (and the corresponding integrals coincide).
\end{cor}

\section{MC-filling families versus the McShane integral}\label{SectionMeasureFilling}

The connection between MC-filling families (Definition~\ref{DefinitionMCfilling}) and the MC-integral is
explained in Proposition~\ref{LemTheScalNullBIS} below. First, it is convenient to characterize MC-filling families as follows:

\begin{lem}\label{McShanefillinglemma}
An hereditary family $\mathcal{F} \subset [\Omega]^{<\omega}$ is MC-filling on~$\Omega$ if (and only if) there exists $\eps>0$ such that for every countable partition $(\Omega_m)$ of $\Omega$ and sets $A_m\in\Sigma$ with $\Omega_m\subset A_m$, there is $F\in\mathcal{F}$ such that
$$
    \mu\left( \bigcup\{A_m : \, F\cap \Omega_m\neq\emptyset \}\right) > \varepsilon.
$$
\end{lem}
\begin{proof} The ``only if'' is obvious. For the converse, we will prove that the condition of Definition~\ref{DefinitionMCfilling}
holds for $0<\eta<\varepsilon$. Suppose we are given a countable partition $(\Omega_m)$ of $\Omega$. For every finite set $I\subset \Nat$, we choose $B_I\in \Sigma$ such that $B_I\supset \bigcup_{m\in I}\Omega_m$ and $\mu(B_I) -\mu^\ast(\bigcup_{m\in I}\Omega_m)<\varepsilon-\eta$. For each $m\in \Nat$, we define $A_m := \bigcap\{ B_I : m\in I\}$. We have $\Omega_m\subset A_m\in\Sigma$, so we can apply the hypothesis to find $F\in\mathcal{F}$ such that
$$
    \mu\left( \bigcup\{A_m : \, F\cap \Omega_m\neq\emptyset \}\right) > \varepsilon.
$$
Consider the (finite) set $I:=\{m \in \Nat: F \cap \Omega_m\neq \emptyset\}$. Since $\bigcup_{m\in I}A_m \subset B_I$, we have
$$
    \mu^\ast\Bigl( \bigcup_{m\in I}\Omega_m\Bigr) >
    \mu(B_I) - (\epsilon-\eta)
    \geq \mu\Bigl( \bigcup_{m\in I}A_m \Bigr) - (\eps-\eta) >\eta.
$$
This proves that $\mathcal{F}$ is MC-filling.
\end{proof}

A set $\Lambda\subset B_{X^*}$ is called \emph{norming} if
$\|x\|=\sup\{|x^*(x)|:x^*\in \Lambda\}$ for all $x\in X$. 

\begin{prop}\label{LemTheScalNullBIS}
Let $f:\Omega \to X$ be a function for which there exist a norming set $\Lambda\subset B_{X^*}$ and a
family $(C_{x^*})_{x^*\in \Lambda}$ of subsets of~$\Omega$ such that $x^*f=1_{C_{x^*}}$ and $\mu^*(C_{x^*})=0$
for every $x^*\in \Lambda$. The following statements are equivalent:
\begin{enumerate}
\item[(i)] $f$ is not MC-integrable;
\item[(ii)] $\bigcup_{x^*\in \Lambda}[C_{x^*}]^{<\omega}$ is MC-filling on~$\Omega$.
\end{enumerate}
\end{prop}
\begin{proof} Observe first that for every finite family $(E_i)$ of pairwise disjoint elements of~$\Sigma$ and
every choice of points $t_i\in \Omega$, we have
\begin{multline}\label{EqnRiemannSumSupremum}
    \Bigl\|\sum_i\mu(E_{i})f(t_i)\Bigr\|
    = \sup_{x^*\in \Lambda} \, \Bigl|x^*\Bigl(\sum_i \mu(E_{i})f(t_i)\Bigr)\Bigr|
    = \\ =\sup_{x^*\in \Lambda} \, \sum_i \mu(E_{i})1_{C_{x^*}}(t_i)
    = \sup_{x^*\in \Lambda} \, \mu\left(\bigcup\{E_{i} : \, t_i\in C_{x^*}\}\right).
\end{multline}
Since $\Lambda$ separates the points of~$X$ and $x^*f$ vanishes $\mu$-a.e. for each $x^*\in \Lambda$, 
the MC-integral of~$f$ is~$0 \in X$ whenever $f$ is MC-integrable. Bearing in mind~\eqref{EqnRiemannSumSupremum}, statement~(i) is equivalent to:
\begin{enumerate}
\item[(iii)] \emph{There exists $\eps>0$ such that for every $\eta>0$, every countable partition $(\Omega_m)$ of~$\Omega$ and 
every sets $A_m\in\Sigma$ with $\Omega_m\subset A_m$, there exist an $\eta$-thick finite family
$(E_i)$ of pairwise disjoint elements of~$\Sigma$ with $E_i\subset A_{m(i)}$, points $t_i\in\Omega_{m(i)}$ and a functional $x^*\in \Lambda$ such that
$\mu\left(\bigcup\{E_{i} : t_i\in C_{x^*}\}\right)>\epsilon$.}
\end{enumerate}

Let us turn to the proof of (iii)$\Leftrightarrow$(ii). Assume first that (iii) holds and take a countable partition $(\Omega_m)$ of~$\Omega$ and sets $\Omega_m \subset A_m \in \Sigma$. Choose $\eta>0$ arbitrary and let $(E_i)$, $(t_i)$ and $x^*$ be as in~(iii).
Observe that the set $F$ made up of all $t_i$'s belonging to~$C_{x^*}$ satisfies
$$
    \bigcup \{A_m: \, F\cap \Omega_m\neq \emptyset\}\supset \bigcup \{E_i: \, t_i\in C_{x^*}\}
$$
and so $\mu(\bigcup\{A_m: F\cap \Omega_m\neq\emptyset\}) > \varepsilon$. According to Lemma~\ref{McShanefillinglemma}, this proves that the family $\bigcup_{x^*\in \Lambda}[C_{x^*}]^{<\omega}$ is MC-filling on~$\Omega$.

Conversely, assume that (ii) holds. Let~$\epsilon>0$ be as in Lemma~\ref{McShanefillinglemma} applied to the family $\bigcup_{x^*\in \Lambda}[C_{x^*}]^{<\omega}$.
Fix $\eta>0$, a countable partition $(\Omega_m)$ of~$\Omega$ and sets $A_m\in\Sigma$ with $\Omega_m\subset A_m$. There exist $x^*\in \Lambda$ and 
finite $F\subset C_{x^*}$ such that
$\mu(\bigcup_{m\in I}A_m)>\epsilon$,
where $I:=\{m\in \Nat: F\cap \Omega_m\neq\emptyset\}$.
Now take a finite set $J\subset \Nat$ disjoint from~$I$ such that $(A_m)_{m\in I\cup J}$ is $\eta$-thick.
Enumerate $I=\{m(1),\dots,m(n)\}$ and $J=\{m(n+1),\dots,m(k)\}$.
Set $E_1:=A_{m(1)}$ and $E_i:=A_{m(i)} \setminus \bigcup_{j=1}^{i-1}A_{m(j)}$ for $i=2,\dots,k$.
Then $(E_i)$ is an $\eta$-thick finite family of pairwise disjoint elements of~$\Sigma$
with $E_i \subset A_{m(i)}$ and $\bigcup_{i=1}^n E_i=\bigcup_{m\in I}A_m$.
Choose $t_i\in F \cap \Omega_{m(i)}$ for $i=1,\dots,n$
and choose $t_i\in \Omega_{m(i)}$ arbitrary for $i=n+1,\dots,k$. Then
$$
    \mu\Bigl(\bigcup\{E_{i} : \, t_i\in C_{x^*}\}\Bigr)
    \geq  \mu\Bigl(\bigcup_{i=1}^n E_{i}\Bigr)=
    \mu\Bigl(\bigcup_{m\in I}A_m\Bigr) >\epsilon.
$$
This shows that (iii) holds, that is, $f$ is not MC-integrable.
\end{proof}

Given a compact Hausdorff topological space~$K$, we write $C(K)$ to denote
the Banach space of all real-valued continuous functions on~$K$ with the sup norm.

\begin{prop}\label{thescalnullfunction}
Let $\mathcal{F} \subset [\Omega]^{<\omega}$ be a compact hereditary family made up of sets having outer measure~$0$.
Let $f:\Omega\to C(\mathcal{F})$ be defined by $f(t)(F) := 1_F(t)$, $t\in \Omega$,
$F\in \cF$. Then:
\begin{enumerate}
\item[(i)] $f$ is scalarly null;
\item[(ii)] $f$ is not MC-integrable if and only if $\mathcal{F}$ is MC-filling on~$\Omega$.
\end{enumerate}
\end{prop}
\begin{proof}
Part (i) follows from a standard argument which we include for the sake of completeness. Since $\mathcal{F}$ is an Eberlein compact (i.e., it is homeomorphic to a weakly compact subset of a Banach space), the space $C(\mathcal{F})$ is WCG and $B_{C(\cF)^*}$ is an
Eberlein compact when equipped with the $w^*$-topology, cf. \cite[Theorem~4, p.~152]{die3}. Set $\Lambda:=\{\delta_F: F\in \mathcal{F}\} \subset B_{C(\cF)^*}$, where $\delta_F$ denotes the ``evaluation functional'' at~$F$. Since $\Lambda$ is norming, its absolutely convex hull ${\rm aco}(\Lambda)$ is $w^*$-dense in~$B_{C(\cF)^*}$. Bearing in mind that $(B_{C(\cF)^*},w^*)$ is homeomorphic to a weakly compact subset of a Banach space, the Eberlein-Smulyan theorem (cf. \cite[Theorem~3.10]{flo}) ensures that ${\rm aco}(\Lambda)$ is $w^*$-sequentially dense in~$B_{C(\cF)^*}$. Since the composition $\delta_F f=1_F$ vanishes $\mu$-a.e. for every $F\in \cF$, we conclude that $f$ is scalarly null. 

Part (ii) follows from Proposition~\ref{LemTheScalNullBIS} applied to $f$ and to 
the norming set~$\Lambda$ defined above.
\end{proof}

On the other hand, it turns out that we can find compact MC-filling families on~$[0,1]$. In order to establish this, we need the following lemma due to Fremlin \cite[4B]{freDU}. Fremlin's result is more general, but this statement is enough for our purposes. We include an elegant proof of it that has been communicated to us by Jordi L\'{o}pez-Abad. Let $\mathfrak{c}$ denote the cardinality of the continuum.

\begin{lem}[Fremlin]\label{logfillingfamily}
There exists a compact hereditary family $\DD\subset [\mathfrak{c}]^{<\omega}$ such that for every infinite $P\subset \mathfrak{c}$ and every $n\in \Nat$ there is $D\in\DD$ such that $D\subset P$ and $|D|=n$.
\end{lem}

\begin{proof}
The so-called Schreier family $\mathcal{S} = \{S\subset\mathbb{N} : |S| \leq \min(S)\}$ is a compact $\frac{1}{2}$-filling 
family on~$\Nat$ (notice that for any $F \in [\mathbb{N}]^{<\omega}$ with $|F|$ even (resp. odd),
the last $\frac{|F|}{2}$ (resp. $[\frac{|F|}{2}]+1$) elements of~$F$ 
form a member of~$\mathcal{S}$). 
Let $T = 2^{<\omega}$ be the dyadic tree, the set of all finite sequences of $0$'s and $1$'s, 
endowed with the tree order: $(x_1,\ldots,x_n)\preceq (y_1,\ldots,y_m)$ if $n\leq m$ and $x_i=y_i$ for $i\leq n$. 
Let $\mathcal{S}'$ be a compact $\frac{1}{2}$-filling family on~$T$, 
that we can obtain by transferring the Schreier family~$\mathcal{S}$ through a bijection between $\mathbb{N}$ and $T$. Given two different $x=(x_1,x_2,\ldots)$ and $y=(y_1,y_2,\ldots)$ 
in $2^\mathbb{N}$, we define $m(x,y) := \min\{k : x_k\neq y_k\}$ and 
$$
	v(x,y) := (x_1,\ldots,x_{m(x,y)-1}) = 
	(y_1,\ldots,y_{m(x,y)-1})\in T.
$$
Given a set $D\subset 2^\mathbb{N}$, we consider $v(D) := \{ v(x,y) : x,y\in D, x\neq y\}$. The family
$$
	\mathcal{D} := \bigl\{ D\in \bigl[2^\mathbb{N}\bigr]^{<\omega} : \, v(D)\in\mathcal{S}'\bigr\}
$$
is obviously hereditary and compact. Indeed, take any
set $A \subset 2^{\Nat}$ with $[A]^{<\omega}\subset \mathcal{D}$. Then 
$[v(A)]^{<\omega}\subset \mathcal{S}'$ and the compactness of $\mathcal{S}'$
ensures that $v(A)$ is finite, and hence, so is~$A$.

We shall prove that for every finite set $A\subset 2^\mathbb{N}$ with $|A|\geq 2$ there is $D\in\DD$ such that
$D\subset A$ and $|D|> \frac{1}{2}\log_2(|A|-1)+1$. This implies the property of $\mathcal{D}$ stated in the lemma. 
To this end, we first prove by induction on~$k\in\Nat \cup \{0\}$ that:
\begin{itemize}
\item [(*)] If $C\subset T$ is a finite set satisfying $|C|\geq 2^k$ and $\inf\{s,s'\}\in C$ whenever $s,s'\in C$, 
then $C$ contains a totally ordered subset of cardinality~$k+1$.
\end{itemize}
Indeed, let $C\subset T$ be a finite set such that $|C|\geq 2^{k+1}$ and 
that $\inf\{s,s'\}\in C$ whenever $s,s'\in C$. If we denote $t:=\min C$, 
then $C\setminus\{t\}=B_0\cup B_1$, where
$B_i$ consists of those elements of $C$ extending $t^\smallfrown i$ for $i=0,1$.
Clearly, we have $\inf\{s,s'\}\in B_i$ whenever $s,s'\in B_i$.
For some $j\in \{0,1\}$ we have $|B_j| \geq 2^k$, so by the inductive 
assumption $B_j$ contains a totally ordered subset $S$ with $|S|=k+1$.
Therefore, $\{t\} \cup S$ is a totally ordered subset of~$C$ 
with cardinality~$k+2$. This finishes the proof of~(*). 

Now, fix a finite set $A\subset 2^\mathbb{N}$ with $|A|\geq 2$. It is easy to check that 
$\inf\{s,s'\} \in v(A)$ for every $s,s' \in v(A)$ and that $|v(A)| = |A|-1$. Let
$k$ be the largest integer less than or equal to $\log_2(|v(A)|)$.
Property~(*) ensures the existence of a totally ordered set $U \subset v(A)$ with $|U|=k+1$,
hence 
$$
	|U|> \log_2(|v(A)|) = \log_2(|A|-1).
$$ 
Since $\mathcal{S'}$ is $\frac{1}{2}$-filling, there exists $W\subset U$ such that $W\in\mathcal{S}'$ and 
$$
	|W|\geq \frac{1}{2}|U|.
$$ 
Let us write $W = \{w_1 \prec w_2 \prec \cdots \prec w_m\}$. 
For each $i=1,\ldots,m-1$, since $w_i\in v(A)$, we can choose $a_i\in A$ such that 
$w_i$ is an initial segment of $a_i$ but $w_{i+1}$ is not. We can also choose $a_m,a_{m+1}\in A$ with $v(a_m,a_{m+1}) = w_m$. 
Set $D := \{a_1,\ldots,a_{m+1}\}$. Notice that $v(D) = W \in\mathcal{S}'$, hence $D\in\mathcal{D}$. Finally, 
$$
	|D| = |W| + 1 \geq \frac{1}{2}|U|+1 > \frac{1}{2}\log_2(|A|-1)+1.
$$
The proof is over.
\end{proof}

\begin{theo}\label{PropositionGrzegorz}
There exists a compact MC-filling family on~$[0,1]$ equipped with the Lebesgue measure.
\end{theo}
\begin{proof}  We denote by~$\lambda$ the Lebesgue measure on~$[0,1]$. Fix a partition $\{Z_\alpha:\alpha<\mathfrak{c}\}$ of $[0,1]$
made up of sets of outer measure one (cf. \cite[419I]{FreMT-4}). Let $\vf:[0,1]\to\mathfrak{c}$ be the
function defined by $\vf(t)=\alpha$ whenever $t\in Z_\alpha$. Let $\DD \subset [\mathfrak{c}]^{<\omega}$ 
be the family provided by Lemma~\ref{logfillingfamily} and define
$$
    \FF:=\bigl\{F \subset [0,1] \mbox{ finite}: \, \vf \mbox{ is one-to-one on }F \mbox{ and }\vf(F)\in\DD\bigr\}.
$$
We claim that $\cF$ is compact. Indeed, take a
set $A\subset [0,1]$ with $[A]^{<\omega} \subset \FF$. Observe that $\vf$ is one-to-one on~$A$. Given any $C \in [\vf(A)]^{<\omega}$, we have $C=\vf(B)$ for some $B\in [A]^{<\omega} \subset \FF$ and so $C\in \DD$. Hence $[\vf(A)]^{<\omega} \subset \DD$ and the compactness of~$\DD$ ensures that $\vf(A)$ is finite. Since $\varphi$ is one-to-one on~$A$, we conclude that $A$ is finite.
This shows that $\cF$ is compact, as claimed.

We shall check that $\FF$ is $\epsilon$-MC-filling on $[0,1]$ with an arbitrary constant $0<\epsilon<1$.
Fix a countable partition $(\Omega_m)$ of~$[0,1]$. For each $\alpha<\mathfrak{c}$ we have
$$
    1=\lambda^*(Z_\alpha)=\lim_{n\to\infty} \lambda^*\Bigl(Z_\alpha\cap \bigcup_{m=1}^n \Omega_m\Bigr),
$$
so we can pick $k(\alpha)\in \Nat$ such that
\begin{equation}\label{EqG}
    \lambda^*\Bigl(Z_\alpha\cap \bigcup_{m=1}^{k(\alpha)} \Omega_m\Bigr)> \epsilon.
\end{equation}
Fix $n\in \Nat$ such that $P_n:=\{\alpha<\mathfrak{c}:k(\alpha)=n\}$ is infinite. By Lemma~\ref{logfillingfamily},
there is $D\in\DD$ such that $D\subset P_n$ and $|D|=n$. Write $D=\{\alpha_1,\dots,\alpha_n\}$.

We next define $t_j\in Z_{\alpha_j}$ and $m_j\in \{1,\dots,n\}$ (with $m_j\neq m_{i}$ whenever $j\neq i$)
inductively as follows. By~\eqref{EqG} the set $Z_{\alpha_1}\cap  \bigcup_{m=1}^n \Omega_m$ is nonempty. Pick then any
$t_1\in Z_{\alpha_1}\cap  \bigcup_{m=1}^n \Omega_m$.
Choose $m_1 \in \{1,\dots,n\}$ so that $t_1\in \Omega_{m_1}$. Now suppose we have already constructed
a set $\{m_1,\dots,m_l\} \subset \{1,\dots,n\}$ and points $t_j\in Z_{\alpha_j}\cap \Omega_{m_j}$
for $j=1,\dots,l$. If $\lambda^*(\bigcup_{j=1}^l \Omega_{m_j})> \epsilon$, the construction stops. Otherwise
$\lambda^*(\bigcup_{j=1}^l \Omega_{m_j})\leq \epsilon$ and therefore
$l < n$ (bear in mind that $\lambda^*(\bigcup_{m=1}^n \Omega_m)> \epsilon$ by~\eqref{EqG}). Writing
$N:=\{1,\dots,n\}\setminus \{m_1,\dots,m_l\}$,
another appeal to~\eqref{EqG} yields
$$
    \lambda^*\Bigl(Z_{\alpha_{l+1}}\cap\bigcup_{m\in N}\Omega_{m}\Bigr)\ge
    \lambda^*\Bigl(Z_{\alpha_{l+1}}\cap\bigcup_{m=1}^n\Omega_m\Bigr)-
    \lambda^*\Bigl(Z_{\alpha_{l+1}}\cap\bigcup_{j=1}^l\Omega_{m_j}\Bigr)> 0,
$$
so we can find $t_{l+1}\in Z_{\alpha_{l+1}}\cap \Omega_{m_{l+1}}$ for some $m_{l+1}\in N$. Repeating the process,
the construction stops for some $l\in \{1,\dots,n\}$.

After that, we obtain a set
$\{m_1,\dots,m_l\} \subset \{1,\dots,n\}$ with $\lambda^*(\bigcup_{j=1}^l \Omega_{m_j})> \epsilon$ and points $t_j\in Z_{\alpha_j}\cap \Omega_{m_j}$
for all~$j=1,\dots,l$. Putting $F:=\{t_1,\dots,t_l\}$ we have
$$
    \lambda^*\Bigl(\bigcup \{\Omega_m: \, F \cap \Omega_m \neq\emptyset\}\Bigr)=
    \lambda^*\Bigl(\bigcup_{j=1}^l \Omega_{m_j}\Bigr) >\epsilon.
$$
Since $\vf(t_j)=\alpha_j$ for all~$j$, it follows that $\vf$ is one-to-one on~$F$ and
$\vf(F)\subset D$, thus $\vf(F) \in \DD$ and so $F\in \FF$. We proved that
the family $\FF$ is $\epsilon$-MC-filling.
\end{proof}

It was kindly communicated to us by Marian Fabian (during the 25th Summer Conference on Topology
and its Applications, Kielce, July 2010) 
that the proof of Theorem~\ref{PropositionGrzegorz} still works if we assume that 
$\mathcal{D}$ has the following property, weaker than the one stated in Lemma~\ref{logfillingfamily}:
\begin{itemize}
\item $\mathcal{D}$ is a compact hereditary family of countable subsets of $\mathfrak c$ such that, for every countable decomposition $\mathfrak{c} = \bigcup_n \Gamma_n$, there exist $D\in\mathcal{D}$ and $n\in\mathbb{N}$ such that $|D\cap \Gamma_n|>n$.
\end{itemize}
A family $\mathcal{D}$ with this property is the same as an Eberlein compact which is not uniformly Eberlein compact subset of the $\Sigma$-product $\Sigma(2^\mathfrak{c})$, cf. \cite{fab-alt-J-2} for further reference.

We now arrive at our main result (valid in ZFC):

\begin{theo}\label{TheoremDP}
There exist a WCG Banach space~$X$ and a scalarly null function
$f:[0,1]\to X$ which is not McShane integrable.
\end{theo}
\begin{proof}
By Theorem~\ref{PropositionGrzegorz}, there is
a compact MC-filling family $\cF$ on~$[0,1]$.
As we observed in the proof of Proposition~\ref{thescalnullfunction}, the space $X:=C(\cF)$ is WCG.
The function $f:[0,1]\to C(\mathcal{F})$ defined by
$$
    f(t)(F):=1_F(t), \quad t\in [0,1], \ F \in \FF,
$$
from Proposition~\ref{thescalnullfunction} is scalarly null and fails to be MC-integrable. According to Proposition~\ref{McimpliesMc2},
$f$ is not McShane integrable.
\end{proof}

Moreover, the target space can be taken reflexive:

\begin{theo}\label{CorollaryReflexive}
There exist a reflexive Banach space~$Y$ and a scalarly null function
$g:[0,1]\to Y$ which is not McShane integrable.
\end{theo}
\begin{proof}
Let $\cF$ and~$f$ be as in the proof of Theorem~\ref{TheoremDP}. Observe first that $f([0,1])$ is relatively weakly compact in~$C(\cF)$. Indeed, by
the Eberlein-Smulyan theorem (cf. \cite[3.10]{flo}), it is enough to check that $(f(t_n))$ converges weakly to~$0$ whenever $(t_n)$ is a sequence of distinct points of~$[0,1]$. But this follows directly from Grothendieck's theorem (cf. \cite[4.2]{flo}) just bearing in mind that for each $F\in \cF$ (finite!) we have $f(t_n)(F)=1_F(t_n)=0$ for $n$ large enough.

Then, by the Davis-Figiel-Johnson-Pelczynski theorem (cf. \cite[Chapter~5, \S4]{die3}), there exist a reflexive Banach space~$Y$ and a one-to-one linear continuous mapping $T:Y \to C(\cF)$ such that $f([0,1])\subset T(B_Y)$.  The set of compositions
$$
    V:=\{\phi\circ T: \, \phi\in C(\cF)^*\}
$$
is a linear subspace of~$Y^*$ which separates the points of~$Y$ (because $T$ is one-to-one). Since $Y$ is reflexive, $V$ is norm dense in~$Y^*$.
Let $g:[0,1] \to Y$ be the function satisfying $T\circ g=f$. For each $y^*\in V$ the composition $y^*g$ vanishes a.e. ($f$ is
scalarly null). This fact and the norm density of~$V$ imply that $g$ is scalarly null. Moreover, since $f$ is not McShane integrable and $T$ is linear and continuous, $g$ is not McShane integrable either.
\end{proof}

\begin{rem}\label{RemarkUniversal}
\rm A glance at the proof of Proposition~\ref{thescalnullfunction}
reveals that the function $f$ from Theorem~\ref{TheoremDP} satisfies that, for each $x^*\in X^*$, the composition
$x^*f$ vanishes up to a countable set. This property and the boundedness of~$f$ ensure that $f$ is \emph{universally Pettis integrable}, that is, Pettis integrable with respect to any Radon probability on~$[0,1]$. The same holds true for the function~$g$ from Theorem~\ref{CorollaryReflexive}. Thus, we answer Question~2.2 in~\cite{rodAUS}: there exist ZFC examples of universally Pettis integrable functions which are not universally McShane integrable.
 \end{rem}

\section{Filling versus MC-filling families}\label{SectionFilling}

In this section we prove that $\varepsilon$-filling families (Definition~\ref{DefinitionFilling}) on sets of positive outer measure are MC-filling. This result is less powerful than Theorem~\ref{PropositionGrzegorz}, in the sense that the existence of compact $\varepsilon$-filling families on uncountable sets is unknown while Theorem~\ref{PropositionGrzegorz} is a ZFC result. Yet, we have decided to include it as it may have some interest in relation with problem~DU.

\begin{theo}\label{epsilonetafilling}
Suppose $\mu$ is atomless. Let $A\subset\Omega$ with $\mu^*(A)>0$ and
$\mathcal{F} \subset [A]^{<\omega}$ be a family which is $\varepsilon$-filling on~$A$ for some $\epsilon>0$.
Then $\mathcal{F}$ is MC-filling on $\Omega$.
\end{theo}
\begin{proof}
Denote $\eta:=\mu^\ast(A)$ and fix $\eta>\eta_1> \eta_2>0$. Take any countable partition $(\Omega_m)$ of~$\Omega$ and sets $A_m\supset \Omega_m$
with $A_m\in \Sigma$.
We will prove that there is $F\in\mathcal{F}$ such that
$$
    \mu\left( \bigcup\{A_m : \, F\cap \Omega_m\neq\emptyset\} \right) > \varepsilon(\eta-\eta_1).
$$
According to Lemma~\ref{McShanefillinglemma}, this means that $\mathcal{F}$ is MC-filling on~$\Omega$.

To this end, take $m_0\in \Nat$ large enough such that
\begin{equation}\label{EqnApproxOuter}
    \mu^*(A)-\mu^*\Bigl(A\cap \bigcup_{m=1}^{m_0}\Omega_m\Bigr)<\eta_2.
\end{equation}
Since $\mu$ is atomless, every finite subset of~$\Omega$ has outer measure~$0$, so we can assume without loss of generality that $A\cap \Omega_m$ is infinite for all $m=1,\dots,m_0$. Take $0<\eta_3<(\eta_1-\eta_2)/m_0$.

We can find pairwise disjoint $B_1,\dots,B_{m_0}\in \Sigma$ such that $\bigcup_{m=1}^{m_0}B_m=\bigcup_{m=1}^{m_0}A_m$
and $B_m\subset A_m$. Let $M$ be the set of all $m\in \{1,\dots,m_0\}$ for which $\mu(B_m)>0$.
For each $m\in M$, choose a positive rational $\alpha_m$ such that
$$
    \mu(B_m)>\alpha_m > \mu(B_m)-\eta_3.
$$
We can write $\alpha_m=p_m/q$ for some $p_m \in \Nat$ and $q\in \Nat$, for $m=1,\dots,m_0$. Set $\theta:=1/q$. Since $\mu$ is atomless, for each $m\in M$ we can find pairwise disjoint $E_1^m,\dots,E_{p_m}^m\in \Sigma$ contained in~$B_m$ with $\mu(E_i^m)=\theta$. Then
$$
    \mu\Bigl(B_m\setminus \bigcup_{i=1}^{p_m}E_i^m\Bigr)<\eta_3
$$
and we have
\begin{multline*}
    \mu^*\Bigl(A\cap \bigcup_{m=1}^{m_0}\Omega_m \Bigr)\leq \mu^*\Bigl(A\cap \bigcup_{m=1}^{m_0}A_m \Bigr)=
    \mu^*\Bigl(A\cap \bigcup_{m\in M}B_m \Bigr) \leq  \\ \leq
        \sum_{m\in M}\mu\Bigl(B_m\setminus \bigcup_{i=1}^{p_m}E_i^m\Bigr)+
        \sum_{m\in M}\sum_{i=1}^{p_m}\mu(E_i^m) \leq \\
        \leq |M| \eta_3 + \Bigl(\sum_{m\in M} p_m\Bigr)\theta \leq m_0 \eta_3 + \Bigl(\sum_{m\in M} p_m\Bigr)\theta <
        (\eta_1-\eta_2) + \Bigl(\sum_{m\in M} p_m\Bigr)\theta.
\end{multline*}
From these inequalities and~\eqref{EqnApproxOuter} we obtain
\begin{equation}\label{EqnMarras}
    \eta=\mu^*(A)< \eta_1 + \Bigl(\sum_{m\in M} p_m\Bigr)\theta.
\end{equation}

For each $m\in M$ and $i=1,\dots,p_m$ we pick a point $t_{(m,i)}\in A \cap \Omega_m$. This can be done in such a way that
the points $t_{(m,i)}$'s are different, since $A\cap \Omega_m$ is infinite for all $m\in M$.
Now $H:=\{t_{(m,i)} : \, m\in M, \, i=1,\dots,p_m\}$ is a subset of~$A$ with cardinality~$\sum_{m\in M} p_m$. Since $\mathcal{F}$ is $\epsilon$-filling on~$A$, there exists
$F \subset H$ with $F \in \mathcal{F}$ such that $|F|\geq \varepsilon \sum_{m\in M} p_m$.
By \eqref{EqnMarras}, we get
\begin{multline*}
    \mu\Bigl(\bigcup\{A_m : \, F\cap \Omega_m\neq\emptyset\} \Bigr)
    \geq \\ \geq
    \mu\Bigl(\bigcup \{E_i^m: \, t_{(m,i)}\in F\}\Bigr) =
    |F|\theta \geq \varepsilon \Bigl(\sum_{m\in M} p_m\Bigr)\theta > \eps (\eta-\eta_1).
\end{multline*}
The proof is over.
\end{proof}

\section{McShane integrability versus MC-integrability}\label{SectionRemarks}

This section is devoted to ensure the existence of McShane integrable functions which are not MC-integrable
(Theorem~\ref{PropositionMcShaneNoMcShane2}). The proof is divided into several auxiliary lemmas. The first one translates the problem into the language of MC-filling families.

\begin{lem}\label{Propc0NoMcShane2Equivalence}
Let $\Gamma$ be a set. The following statements are equivalent:
\begin{enumerate}
\item[(i)] there exists a scalarly null function $f:\Omega \to c_0(\Gamma)$ which is not MC-integrable and satisfies $f(\Omega)\subset
\{e_\gamma: \gamma \in \Gamma\}$, where $e_\gamma(\gamma')=\delta_{\gamma,\gamma'}$ (the Kronecker symbol) for all $\gamma,\gamma'\in \Gamma$;
\item[(ii)] there exists a partition $(C_\gamma)_{\gamma\in \Gamma}$ of~$\Omega$ into sets having outer measure~$0$ such that the family
$\bigcup_{\gamma\in \Gamma}[C_\gamma]^{<\omega}$ is MC-filling on~$\Omega$.
\end{enumerate}
\end{lem}
\begin{proof} The set $\Lambda:=\{e_\gamma^*: \gamma\in \Gamma\}\subset B_{c_0(\Gamma)^*}$ is norming, where $e_\gamma^*(x)=x(\gamma)$ for all $x\in c_0(\Gamma)$ and $\gamma \in \Gamma$.

(i)$\impli$(ii) For each $\gamma\in \Gamma$ we have $e_\gamma^*f=1_{C_\gamma}$, where
$C_\gamma:=\{t\in \Omega: f(t)=e_\gamma\}$ has outer measure~$0$ (because $f$ is scalarly null). Clearly, $(C_\gamma)_{\gamma\in \Gamma}$ is a partition of~$\Omega$. Since $f$ is not MC-integrable, we can apply Proposition~\ref{LemTheScalNullBIS} to conclude that the family
$\bigcup_{\gamma\in \Gamma}[C_\gamma]^{<\omega}$ is MC-filling on~$\Omega$.

(ii)$\impli$(i) Define $f:\Omega \to c_0(\Gamma)$ by $f(t):=e_\gamma$ whenever $t\in C_\gamma$, $\gamma\in \Gamma$. Then
$e_\gamma^*f=1_{C_\gamma}$ for all $\gamma\in \Gamma$ and $f$ is scalarly null, because $\mu^*(C_\gamma)=0$ for all $\gamma\in \Gamma$ and the linear span of $\{e_\gamma^*:\gamma\in \Gamma\}$ is norm dense in $c_0(\Gamma)^*=\ell^1(\Gamma)$. By Proposition~\ref{LemTheScalNullBIS}, since $\bigcup_{\gamma\in \Gamma}[C_\gamma]^{<\omega}$ is MC-filling on~$\Omega$, the function $f$ is not MC-integrable.
\end{proof}

Thus, bearing in mind that Pettis and McShane integrability are equivalent for $c_0(\Gamma)$-valued functions~\cite{dip-pre}, in order to find McShane integrable functions which are not MC-integrable we will look for MC-filling families like in condition~(ii) of Lemma~\ref{Propc0NoMcShane2Equivalence}. The following sufficient condition will be helpful.

\begin{lem}\label{Lemmac0}
Let $(C_\gamma)_{\gamma\in \Gamma}$ be a partition of~$\Omega$ and $\epsilon>0$ be such that,
whenever $(\Gamma_A)_{A \subset \Nat}$ is a partition of~$\Gamma$, there is some $A \subset \Nat$ such that $\mu^*(\bigcup_{\gamma\in \Gamma_A}C_\gamma)>\epsilon$.
Then the family $\bigcup_{\gamma\in \Gamma}[C_\gamma]^{<\omega}$ is MC-filling on~$\Omega$.
\end{lem}
\begin{proof}
Fix a countable partition $(\Omega_m)$ of~$\Omega$. For each $A\subset \Nat$, set
$$
    \Gamma_A:=\bigl\{\gamma\in \Gamma: \ \{m\in \Nat: C_\gamma \cap \Omega_m \neq \emptyset\}= A\bigr\}.
$$
Then $(\Gamma_A)_{A\subset \Nat}$ is a partition of~$\Gamma$ and so there is $A\subset \Nat$ such that $\mu^*(\bigcup_{\gamma\in \Gamma_A}C_\gamma)>\epsilon$. Observe that
$$
    \bigcup_{m\in A}\Omega_m \supset  \bigcup_{\gamma \in \Gamma_A}C_\gamma,
$$
hence $\mu^*(\bigcup_{m\in A}\Omega_m)>\epsilon$. Choose $B\subset A$ finite with
$\mu^*(\bigcup_{m\in B}\Omega_m)>\epsilon$. Take $\gamma\in \Gamma_A$. We can find a finite set $F\subset C_\gamma$ such that $F\cap \Omega_m\neq \emptyset$ for every $m\in B$, hence
$$
    \mu^*\Bigl(\bigcup\{\Omega_m: \ F\cap \Omega_m\neq \emptyset\}\Bigr)\geq
    \mu^*\Bigl(\bigcup_{m\in B}\Omega_m\Bigr)>\epsilon.
$$
This shows that $\bigcup_{\gamma\in \Gamma}[C_\gamma]^{<\omega}$ is MC-filling on~$\Omega$.
\end{proof}

We now focus on
$2^\kappa$ (for a cardinal~$\kappa$), which is a Radon probability space when equipped with (the completion of) the usual product probability, cf. \cite[416U]{FreMT-4}.

\begin{lem}\label{LemmaZ}
Let $\kappa$ be an uncountable cardinal, $(A_\alpha)_{\alpha<\kappa}$ a partition of~$\kappa$ into infinite sets and consider, for each $\alpha<\kappa$, the sets
$$
    D_\alpha:=\{x\in 2^\kappa: \ x(\gamma)=0 \mbox{ for all }\gamma\in A_\alpha\}
    \quad
    \mbox{and} \quad
    E_\alpha:=D_\alpha \setminus \bigcup_{\beta < \alpha}D_\beta.
$$
Then $\bigcup_{\alpha \in I}E_\alpha$ has outer measure~$1$ for every uncountable set $I\subset \kappa$.
\end{lem}
\begin{proof}
It suffices to check that $Z\cap (\bigcup_{\alpha \in I}E_\alpha) \neq \emptyset$ whenever $Z$ belongs to the product $\sigma$-algebra of~$2^\kappa$ and has positive measure. Fix a countable set $A\subset \kappa$ such that, for any $z\in Z$, we have
\begin{equation}\label{EqnDepCountCoord}
    \{x \in 2^\kappa: \ x(\gamma)=z(\gamma) \mbox{ for all }\gamma \in A\} \subset Z.
\end{equation}
Since the $A_\alpha$'s are disjoint, the set $J:=\{\alpha<\kappa: A\cap A_\alpha \neq \emptyset\}$ is countable.
Clearly, the $D_\alpha$'s have measure zero (because $A_\alpha$ is infinite) and so $Z\setminus \bigcup_{\alpha \in J}D_\alpha$ has positive measure.
In particular, we can choose $z\in Z\setminus \bigcup_{\alpha \in J}D_\alpha$. Since $J$ is countable and $I$ is not, there is $\beta \in I \setminus J$. We now define an element $x\in 2^\kappa$ by declaring
$$
    x(\gamma):=
    \begin{cases}
    z(\gamma) & \text{if $\gamma\in \bigcup_{\alpha \in J}A_\alpha$}, \\
    0 & \text{if $\gamma \in A_{\beta}$}, \\
    1 & \text{otherwise.}
    \end{cases}
$$
We claim that $x\in Z\cap E_\beta$. Indeed, we have $x\in Z$ by~\eqref{EqnDepCountCoord} (bear in mind that $A\subset \bigcup_{\alpha \in J}A_\alpha$). On the other hand, take any $\alpha < \kappa$ with $\alpha \neq \beta$. If $\alpha \in J$ then
$z\not\in D_\alpha$ and so $x \not \in D_\alpha$ as well. If $\alpha \not \in J$, then $x(\gamma)=1$ for all $\gamma\in A_\alpha$ and so $x\not \in D_\alpha$. It follows that $x\in Z\cap E_\beta$, as claimed. Therefore $Z\cap (\bigcup_{\alpha \in I}E_\alpha) \neq \emptyset$.
\end{proof}

\begin{lem}\label{LemmaAuxiliaryNoMcShane2}
Let $\kappa$ be a cardinal with $\kappa>\mathfrak{c}$. Then there is a partition $(C_\gamma)_{\gamma\in \Gamma}$ of~$2^\kappa$ into sets of measure zero such that, whenever $(\Gamma_A)_{A \subset \Nat}$ is a partition of~$\Gamma$, there is some $A \subset \Nat$ such that $\bigcup_{\gamma\in \Gamma_A}C_\gamma$ has outer measure~$1$.
\end{lem}
\begin{proof}
Let $(A_\alpha)_{\alpha<\kappa}$ be a partition of~$\kappa$ into infinite sets. Clearly, the $E_\alpha$'s of Lemma~\ref{LemmaZ} are pairwise disjoint and have measure zero (since $A_\alpha$ is infinite). We claim that the partition 
$$
    \mathcal{C}:=\Bigl\{E_\alpha: \ \alpha<\kappa\Bigr\}\cup \Bigl\{\{x\}: \ x \in 2^\kappa \setminus \bigcup_{\alpha<\kappa}E_\alpha\Bigr\}
$$
of~$2^\kappa$ satisfies the desired property. 
Indeed, let $(\mathcal{C}_A)_{A\subset \Nat}$ be any partition of~$\mathcal{C}$. Since $\kappa>\mathfrak{c}$, there is some $A\subset \Nat$ such that $\cC_A$ contains uncountably many $E_\alpha$'s. By Lemma~\ref{LemmaZ}, the outer measure of $\bigcup \cC_A$ is~$1$, as required.
\end{proof}

We can now state the aforementioned result:

\begin{theo}\label{PropositionMcShaneNoMcShane2}
Let $\kappa$ be a cardinal with $\kappa>\mathfrak{c}$. Then there is a McShane integrable function
$f: 2^\kappa \to c_0(\Gamma)$ (for some set~$\Gamma$) which is not MC-integrable.
\end{theo}
\begin{proof}
By Lemmas~\ref{Propc0NoMcShane2Equivalence}, \ref{Lemmac0} and \ref{LemmaAuxiliaryNoMcShane2}, there
is a scalarly null function $f: 2^\kappa \to c_0(\Gamma)$ (for some set~$\Gamma$) which is not MC-integrable.
Since $f$ is Pettis integrable, it is also McShane integrable \cite{dip-pre}.
\end{proof}

In \cite{dev-rod} it is proved that Pettis and McShane integrability are equivalent for $X$-valued functions defined on quasi-Radon probability spaces whenever $X$ is Hilbert generated (i.e., there exist a Hilbert space~$Y$ and a linear continuous map $T:Y \to X$ such that $T(Y)$ is dense in~$X$).
Clearly, every Hilbert generated space is WCG. Typical examples of Hilbert generated spaces are the separable ones, $c_0(\Gamma)$ (for any set~$\Gamma$), $L^1(\nu)$ (for any probability measure~$\nu$) and $C(K)$ where $K$ is a uniform Eberlein compact space. 
Moreover, any super-reflexive space embeds into a Hilbert generated space.
For more information on this class of spaces, we refer the reader to \cite{fab-alt-J-4,fab-alt-J-2} and \cite[Chapter~6]{fab-alt-JJ}.

In view of our Theorem~\ref{PropositionMcShaneNoMcShane2}, we cannot replace McShane integrability by MC-integrability in the results of~\cite{
dev-rod}. However, something can be said for a particular class of functions. The following proposition is inspired by \cite[Lemma~3.3]{dev-rod}. 

\begin{prop}\label{LemmaMbasisValued}
Suppose $\mu$ is atomless and $X$ is a subspace of a Hilbert generated Banach space
such that $|\Omega|\leq {\rm dens}(X)$. Let $I \subset B_X$ be a set such that:
\begin{itemize}
\item the linear span of~$I$ is dense in~$X$;
\item for each $x^*\in X^*$, the set $\{x\in I: x^*(x)\neq 0\}$ is countable. 
\end{itemize}
Then any one-to-one function $f:\Omega \to I \subset X$ is scalarly null and MC-integrable.
\end{prop}
\begin{proof}
For each $x^*\in X^*$ the composition $x^*f$ vanishes up to a countable set, which
has outer measure~$0$ (because $\mu$ is atomless). So, $f$ is scalarly null.

Fix $\epsilon>0$. Since $X$ embeds into a Hilbert generated space,
there is a partition $I=\bigcup_{m\in \Nat} I_m$ such that
\begin{equation}\label{ControlUEC}
    \mbox{for all }x^*\in B_{X^*} \mbox{ and all }m\in \Nat,
    \quad |\{x\in I_m: \, \vert
    x^*(x)\vert > \varepsilon\}| \le m,
\end{equation}
see \cite[Theorem~6]{fab-alt-J-2} (cf. \cite[Theorem~6.30]{fab-alt-JJ}). 

For each $m\in \Nat$, define $\Omega_m:=\varphi^{-1}(I_m)$ and choose pairwise disjoint
sets $A_{1,m}, \dots, A_{N(m),m} \in \Sigma$ with
$$
    \mu(A_{n,m})\leq \frac{\epsilon}{2^m m}, \quad
    n=1,\dots,N(m),
$$
and $\Omega_m \subset \bigcup_{n=1}^{N(m)}A_{n,m}$. Set $\Omega_{n,m}:=\Omega_{m}\cap A_{n,m}$ for all $m\in \Nat$ and $n=1,\dots,N(m)$, so that  $(\Omega_{n,m})$ is a countable partition of~$\Omega$.

Fix a finite family $(E_{j})_{j\in J}$ of pairwise disjoint elements of~$\Sigma$
with $E_j \subset A_{n(j),m(j)}$ and choose any points $t_j\in \Omega_{n(j),m(j)}$.
We can and do assume without loss of generality that $t_{j}\neq t_{j'}$
whenever $j\neq j'$. Fix $x^*\in B_{X^{*}}$. Define
$$
    C:=\{j\in J: \ |x^*(f(t_j))|\leq \epsilon\}
$$
and 
$$
    B_m:=\{j\in J: \ t_j \in \Omega_m \mbox{ and }|x^*(f(t_j))|> \epsilon\}
    \quad \mbox{ for all }m\in \Nat.
$$
Observe that $|B_m|\leq m$ for all $m\in \Nat$ (this follows from \eqref{ControlUEC}, the injectivity of~$\varphi$ 
and the fact that $t_j\neq t_{j'}$ whenever $j\neq j'$).
We can write
\begin{equation}\label{BasicDecomposition}
    \sum_{j\in J} \mu(E_j)f(t_j)=\sum_{j\in C} \mu(E_j)f(t_j)
    + \sum_{m\in \Nat} \sum_{j\in B_m} \mu(E_j)f(t_j).
\end{equation}
On one hand
\begin{equation}\label{FirstSummand}
    x^*\Bigl(\sum_{j\in C} \mu(E_j)f(t_j)\Bigr)\leq
    \mu\Bigl(\bigcup_{j\in C}E_j \Bigr)
     \epsilon \leq \epsilon.
\end{equation}
On the other hand, take any $m\in \Nat$ and any $j\in B_m$. Then $t_j\in \Omega_m \cap \Omega_{n(j),m(j)}$
and so $m(j)=m$. Therefore $E_j \subset A_{n(j),m}$ and so $\mu(E_j) \leq \epsilon/(2^m m)$.
Thus
\begin{equation}\label{SecondSummand}
    x^*\Bigl(\sum_{j\in B_m} \mu(E_j)f(t_j)\Bigr)\leq
    \sum_{j\in B_m}\mu(E_j) \leq  |B_m|\frac{\epsilon}{2^m m} \leq \frac{\epsilon}{2^m}.
\end{equation}
From
\eqref{BasicDecomposition}, \eqref{FirstSummand} and~\eqref{SecondSummand} it follows that
$$
    x^*\Bigl(\sum_{j\in J} \mu(E_j)f(t_j)\Bigr) \leq 2\epsilon.
$$
As $x^*\in B_{X^*}$ is arbitrary, we have $\|\sum_{j\in J}
\mu(E_j)f(t_j)\| \leq 2\epsilon$. Hence $f$ is MC-integrable, with MC-integral $0\in X$.
\end{proof}

A set~$I$ satisfying the requirements of Proposition~\ref{LemmaMbasisValued} 
is constructed for instance in \cite[Theorem~2]{fab-alt-J-2}. Also,
for~$I$ we can take any Markushevich basis in~$X$ (cf. \cite[Lemma~5.35]{fab-alt-JJ}).
Similar ideas yield the following result (cf. \cite[Proposition~4.14]{rod5}):

\begin{pro}\label{PropositionReferee}
Suppose $\mu$ is atomless and $X$ is weakly Lindel\"{o}f determined
such that $|\Omega|={\rm dens}(X)$.
Then there exists a one-to-one scalarly null function $f:\Omega \to X$ 
such that the linear span of $f(\Omega)$ is dense in~$X$.
\end{pro}

\subsection*{Acknowledgements}
We thank the referees for valuable suggestions which improved the presentation
of the paper. Thanks are also due to Marian Fabian, David Fremlin and Jordi L\'{o}pez-Abad
for useful comments.

\def\cprime{$'$}\def\polhk#1{\setbox0=\hbox{#1}{\ooalign{\hidewidth
  \lower1.5ex\hbox{`}\hidewidth\crcr\unhbox0}}} \def\cprime{$'$}
\providecommand{\bysame}{\leavevmode\hbox to3em{\hrulefill}\thinspace}
\providecommand{\MR}{\relax\ifhmode\unskip\space\fi MR }
\providecommand{\MRhref}[2]{%
  \href{http://www.ams.org/mathscinet-getitem?mr=#1}{#2}
}
\providecommand{\href}[2]{#2}

\end{document}